\documentclass[reqno, 11pt]{amsart}
\usepackage[top=3cm, bottom=3.5cm, left=2.75cm, right=2.75cm, headsep=0.75cm]{geometry}            
\usepackage{amsmath,amsrefs,amssymb}
\usepackage{pgfplots}
 \usepackage{tikz,float}
\pgfplotsset{compat = newest}

\usepackage{graphicx}
\usepackage[colorlinks=true, pdfborder={0 0 0}]{hyperref}
\usepackage{subcaption}
\usepackage[font=scriptsize]{caption}
\usepackage{mwe}

\DeclareMathOperator{\divergence}{div}

\newtheorem{theorem}{Theorem}[section]
\theoremstyle{plain}

\newtheorem{step}{Step}[section]

\newtheorem{definition}{Definition}[section]

\newtheorem{lemma}{Lemma}[section]

\newtheorem{remark}{Remark}[section]

\newcommand{\hn}{\mathbb{H}^{n}}

\numberwithin{equation}{section}

\allowdisplaybreaks

\begin{document}
\title[Non-linear Heat Equation on the Hyperbolic Space]{Non-linear heat equation on the Hyperbolic space: Global existence and finite-time Blow-up}

\author{Debdip Ganguly}
\address{Debdip Ganguly, Department of Mathematics\\
Indian Institute of Technology Delhi\\
IIT Campus, Hauz Khas, New Delhi, Delhi 110016, India}
\email{debdipmath@gmail.com, debdip@maths.iitd.ac.in}
\author{Debabrata Karmakar}
\address{Debabrata Karmakar, Tata Institute of Fundamental Research\\
Centre For Applicable Mathematics\\
Post Bag No 6503, GKVK Post Office, Sharada Nagar, Chikkabommsandra, Bangalore 560065, India}
\email{debabrata@tifrbng.res.in}
\author{Saikat Mazumdar}
\address{Saikat Mazumdar: Department of Mathematics\\
Indian Institute of Technology Bombay\\
Mumbai 400076, India}
\email{saikat@math.iitb.ac.in, saikat.mazumdar@iitb.ac.in}

\date{\today}

%\thanks{XX}

\begin{abstract}

We consider the following Cauchy problem for the semi linear heat equation on the hyperbolic space :

 \begin{align}\label{abs:eqn}
\left\{\begin{array}{ll}
\partial_{t}u=\Delta_{\mathbb{H}^{n}} u+  f(u, t) &\hbox{ in }~ \mathbb{H}^{n}\times (0, T),\\
\\
\quad u =u_{0} &\hbox{ in }~ \mathbb{H}^{n}\times \{0\}.
\end{array}\right.
\end{align}
We study Fujita phenomena for the non-negative initial data $u_0$ belonging to  $C(\mathbb{H}^{n}) \cap L^{\infty}(\mathbb{H}^{n})$ and for different choices of $f$ of the form $f(u,t) = h(t)g(u).$ It is well-known that for power 
nonlinearities in $u,$ the power weight $h(t) = t^q$ is sub-critical in the sense that non-negative global solutions exist for small initial data. On the other hand,  \eqref{abs:eqn} exhibits
Fujita phenomena for the exponential weight $h(t) = e^{\mu t},$ i.e. there exists a critical exponent $\mu^*$ such that if $\mu > \mu^*$ then all non-negative solutions blow-up in finite time and if $\mu \leq \mu^*$ 
there exists non-negative global solutions for small initial 
data. One of the main objectives of this article is to find an appropriate nonlinearity in $u$ so that \eqref{abs:eqn} with the power weight $h(t) = t^q$ does exhibit Fujita phenomena. In the remaining part of this article, we study Fujita phenomena 
for exponential nonlinearity in $u.$  We 
further generalize some of these results to Cartan-Hadamard manifolds.
 
 %exponential type, in s-variable  as well in time, we show existence of global solutions and blow-up of solutions for certain range of parameters, depending 
%on the bottom of spectrum of $-\Delta_{\mathbb{H}^N}.$ Furthermore we consider logarithmic type nonlinearity for $f(s, t)$ in  s-variable and  power type 
%nonlinearity in time variable and we establish Fujita-type phenomenon for such $f.$  

\end{abstract}
\maketitle
%\tableofcontents

\section{Introduction}

 In this paper we study the global existence and finite time blow-up of solutions to some nonlinear heat equation on the hyperbolic space (the simplest example of Cartan-Hadamard manifolds).  In particular, we consider the following Cauchy problem on the Hyperbolic space $\mathbb{H}^{n}$ of dim $n \geq 2$
 \begin{align}\label{eqn:main}
\left\{\begin{array}{ll}
\partial_{t}u=\Delta_{\mathbb{H}^{n}} u+ f(u, t) &\hbox{ in }~ \mathbb{H}^{n}\times (0, T),\\
\\
\quad u =u_{0}\in C(\mathbb{H}^{n}) \cap L^{\infty}(\mathbb{H}^{n}) &\hbox{ in }~ \mathbb{H}^{n}\times \{0\},
\end{array}\right.
\end{align}
where  $\Delta_{\mathbb{H}^{n}}= \divergence(\nabla \cdot)$ is the Laplace-Beltrami operator on $\mathbb{H}^{n}, u_0 \geq 0$ and  the nonlinearity $f$ is prescribed as follows : 

\begin{enumerate}
\item[(I)] In the first  case,  the nonlinearity is of the form $f(s, t) := t^q g(s)$ where $q > 0$ and  $g : [0, \infty) \rightarrow [0, \infty) $ is a convex function satisfying 
\begin{align}\label{nonlinearity2}
g(s) :=
\left\{\begin{array}{ll}
  s |\ln s|^{-\alpha}, & \mbox{for}  \ s \sim 0 \\  
   h(s),  & \mbox{for} ~ s>  \frac{1}{2},
\end{array}\right.
\end{align}
where $\alpha > 0, h(s) \geq  \kappa s^2,$ for some $\kappa > 0.$ We shall call this {\rm Type-I nonlinearity}.

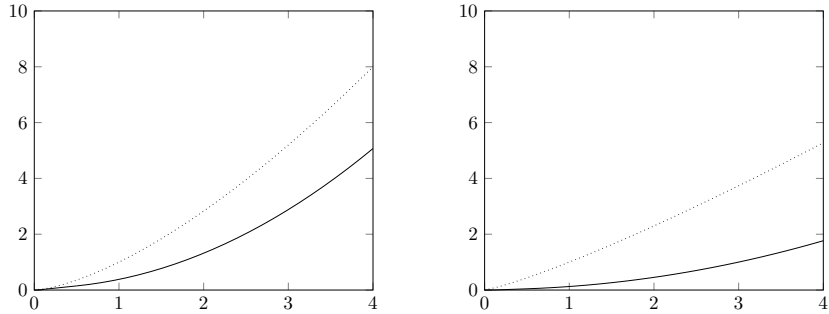
\begin{figure}[H]
  \centering
\begin{subfigure}{0.2\textwidth}
\begin{tikzpicture}[scale=0.65]
 \begin{axis}[
    xmin = 0, xmax = 4,
    ymin = 0, ymax = 10]
    \addplot[
        domain = 0: 1/(exp(1)*exp(1)),
    ] { x/((ln(1/x))*(ln(1/x)))};
   %exp(-x/10)*( cos(deg(x)) + sin(deg(x))/10 )};
     \addplot[
        domain = 1/(exp(1)*exp(1)):0.5,
    ] {(1/4)*(1/(exp(1)*exp(1))) + (5/16)*(x - 1/(exp(1)*exp(1)))};
    \addplot[
        domain = 0.5:4,
    ] {(5/16)*(x*x) + (5/(2^6) - (exp(-2))/(2^4))};
    \addplot[ dotted,
        domain = 0:4,
    ] {x^(1.5)};
    \end{axis}
  \end{tikzpicture}
 \end{subfigure}
 \quad \qquad \qquad \qquad 
  \begin{subfigure}{0.2\textwidth}
 \begin{tikzpicture}[scale=0.65]
 \begin{axis}[
    xmin = 0, xmax = 4,
    ymin = 0, ymax = 10]
    \addplot[
        domain = 0: 1/(exp(1)*exp(1)),
    ] { x/((ln(1/x)^5)};
   %exp(-x/10)*( cos(deg(x)) + sin(deg(x))/10 )};
     \addplot[
        domain = 1/(exp(1)*exp(1)):0.5,
    ] {(1/32)*(exp(-2)) + (7/(2^6))*(x - exp(-2))};
    \addplot[
        domain = 0.5:4,
    ] {(7/(2^6))*(x^2) - (5/(2^6))*(exp(-2)) + (7/(2^8))};
    \addplot[ dotted,
        domain = 0:4,
    ] {x^(1.2)};
    \end{axis}
  \end{tikzpicture}
 \end{subfigure}
 \caption{The above pictures demonstrate the graph of $g$ for $\alpha = 2$ (left) and $\alpha = 5$ (right) respectively. For $s \leq e^{-2},g$ behaves like $s|\ln s|^{-\alpha}$ and for $s>0.5, g$ dominates $\kappa s^2$ for some $\kappa >0.$ For $e^{-2} \leq s \leq 0.5, g$ grows linearly adhering the convexity structure. The dotted lines are the graphs of $s^{p}$ for $p>1$ but close to $1.$}
 \end{figure}
 
\item[(II)] In the second case,  we consider $f : [0, \infty) \times [0, \infty) \rightarrow \mathbb{R}$ defined by
\begin{equation}\label{nonlinearity1}
f(s, t) := e^{\mu t}s\big( e^{\beta s^{p}}-1\big),
\end{equation}
where $\mu,\beta,p > 0$ are constants. We shall call this {\rm Type-II nonlinearity}.
\end{enumerate}

We study the global existence versus finite-time blow-up of solutions to equation \eqref{eqn:main} corresponding to continuous bounded initial data and for the nonlinearities mentioned above.
 In particular,  we study appropriate relations among the parameters $p,\alpha,\mu,q$ separating the dichotomy of global existence vs finite time blow-up. Let us now briefly mention the history and the main motivations behind considering such nonlinearities.

\medskip 

\noindent $\bullet$ \bf Cauchy Problem with type-I nonlinearity\rm. In  \cite{Fujita}, Fujita observed that the following Cauchy problem for the semi-linear heat equation posed in $\mathbb{R}^n$
\begin{align}\label{eqn:main-euc}
\left\{\begin{array}{ll}
\partial_{t}u=\Delta_{\mathbb{R}^{n}} u+ |u|^{p-1}u &\hbox{ in }~ \mathbb{R}^{n}\times (0, T),\\
\\
\quad u =u_{0}\in C(\mathbb{R}^{n}) \cap L^{\infty}_+(\mathbb{R}^{n}) &\hbox{ in }~ \mathbb{R}^{n}\times \{0\},
\end{array}\right.
\end{align}
exhibits an interesting global existence vs finite time blow-up dichotomy.
He proved that there exists a critical value 
$p^* := 1 + \frac{2}{n},$ called { \rm Fujita exponent}, such that the followings hold: 

%where $p>1$ and $u_0(x)$ is a positive  function exhibit the following properties: 

\begin{enumerate}
\item If $1 < p < p^*,$ the problem \eqref{eqn:main-euc}  does not possess nontrivial non-negative global solutions.

\medskip 

\item If $p > p^*$, the problem \eqref{eqn:main-euc} admits a non-negative global solution for {\it small initial data} $u_0.$
 \end{enumerate}

Later it was discovered that the borderline case $p = p^*$ belongs to the blow-up region, see e.g. \cite{Hayakawa, HAL} for references. A variant of \eqref{eqn:main-euc} was studied by Meier in \cite{Meier} and he found  
that if the nonlinearity $|u|^{p-1}u$ in \eqref{eqn:main-euc} is aided by a time dependent weight of the form $t^q |u|^{p-1}u, q > -1,$ then the critical Fujita exponent becomes $p^* = 1 + \frac{2(q+1)}{n}.$ 
Meier also considered the case of bounded domains 
supplemented with Dirichlet boundary condition
\begin{align}\label{eqn:main-bounded}
\left\{\begin{array}{ll}
\partial_{t}u=\Delta_{\mathbb{R}^{n}} u \, + \, t^q \, |u|^{p-1}u &\hbox{in }~ \Omega \times (0, T),\\
\quad u = 0 \quad   &\hbox{on}~ \partial \Omega \times (0, T), \\
\quad u = u_0 \geq 0 &\hbox{in}~  \Omega \times \{ 0\},
\end{array}\right.
\end{align}
 where $\Omega \subset \mathbb{R}^n$ is a bounded domain and $q >-1$. The situation for bounded domains is quite different from that of the whole Euclidean space $\mathbb{R}^n.$ In particular, there always exists a global solution if $u_0$ is 
 sufficiently small. Meier pointed out that in order 
 to exhibit Fujita phenomenon, it is necessary to consider a weight $h(t)$ which is very large at infinity, for example $h(t) = e^{\mu t}, \mu > 0.$ In connection with Meier's result, Bandle-Pozio and Tesei \cite{Bandle}, 
found out that the problem posed on the hyperbolic space exhibits identical Fujita phenomena as that of bounded domains in $\mathbb{R}^n$ except possibly at the borderline case, which is in contrast to the 
Euclidean setting, belongs the global existence region. The results of Bandle et. al. \cite{Bandle} and completed by Wang and Yin \cite{WangYin}  can be stated as follows: consider the 
Cauchy problem 
 \begin{align}\label{eqn:2}
\left\{\begin{array}{ll}
\partial_{t}u=\Delta_{\mathbb{H}^{n}} u+ h(t)u^{p} &\hbox{ in }~ \mathbb{H}^{n}\times (0, T),\\
\\
\quad u =u_{0}\in C(\mathbb{H}^{n}) \cap L^{\infty}_+(\mathbb{H}^{n}) &\hbox{ in }~ \mathbb{H}^{n}\times \{0\},
\end{array}\right.
\end{align}
Then the followings hold:
\begin{enumerate}

\item If $h(t) = t^q$ with $q > -1,$ then for small initial data $u_0$ there exist global solutions to \eqref{eqn:2}.

\item
Let $h(t) = e^{\mu t}, \mu >0$ and set  

\begin{align*}
p_{F}^{*}(\mu):=1 + \frac{\mu}{\lambda_1} := 1+\frac{4\mu}{(n-1)^{2}}, 
\end{align*}
 where $\lambda_1 := \lambda_1(\hn)$ is the bottom of the $L^2$-spectrum of the $-\Delta_{\hn},$ defined by
 \begin{equation*}\label{poin}
\lambda_1= \inf_{u \in H^{1}(\hn) \setminus \{ 0\}} \frac{\int_{\hn}
|\nabla_{\hn} u|^2 \, {\rm d}v_{\hn}}{\int_{\hn} |u|^2 \, {\rm d}v_{\hn}}=\frac{(n-1)^2}{4}.
\end{equation*}

If $p \geq p_{F}^{*}(\mu)$, solutions to \eqref{eqn:2} exist globally for sufficiently small initial data $u_{0}$.
\item
If  $p < p_{F}^{*}(\mu)$, then all non-negative solutions to \eqref{eqn:2} blow-up in finite time. 
\end{enumerate}

%%%
Recently, considerable efforts have been made by several authors (see \cite{Zhang, MMDP, PF1, PF2} and the references in therein) to extend these results to Riemannian manifolds $M,$ satisfying certain curvature bounds.  More generally, 
the question of Fujita phenomena for Porous medium equation has also drawn significant development in this direction. It is beyond our scope to give a detailed account of the results for Porous medium equation. Instead, we refer the interested readers to the manuscripts \cite{Samar, Zhang, GMP3, GMeP0} and the references quoted therein for more detailed treatment in this area.

%Recently there has been an effort to extend these results to Riemannian manifolds $M,$ under suitable curvature bounds,
%and can be found in \cite{MMDP, PF1, PF2} and the references in therein.

\medskip

Our choice of nonlinearity \eqref{nonlinearity2} stems out from the above mentioned results of \cite{Bandle}. We ask the following question: by keeping the nonlinearity as polynomial in $t,$ 
can one find a function $g$ such that 
the equation \eqref{eqn:main} with  $f(s, t) \, = \, t^q \, g(s),$ exhibit Fujita phenomenon for certain range of parameters? Interestingly,  we found out in Theorem~\ref{thm2} that, 
if $g(s)$ is of logarithmic type singularity mentioned above, then \eqref{eqn:main} exhibits Fujita phenomenon.

\medskip

\noindent $\bullet$ \bf Cauchy Problem with type-II nonlinearity\rm. On the other hand, the nonlinearity \eqref{nonlinearity1} is partly motivated from the celebrated Moser-Trudinger (M-T) inequality 
\cite{JM}. Although the (M-T) inequality 
on manifolds, in particular on the hyperbolic space has been studied extensively in recent years, let us briefly recall it on the hyperbolic space. 
Mancini-Sandeep \cite{MS1} and Adimurthi-Tintarev \cite{ADIT} in dimension $2$ and Lu-Tang \cite{LuTang} in higher dimension, proved that 
the following inequality holds true on the hyperbolic space:
\begin{theorem} (\cite{MS1}):
 Let $\mathbb{B}^n$ be the unit open ball in $\mathbb{R}^n,$ endowed with a conformal metric $h = \rho g_{e},$ where 
$g_{e}$ denotes the Euclidean metric and $\rho(x) = (\frac{2}{1-|x|^2})^2$ then

 \begin{equation}\label{mtr}
 \sup_{u \in C_{0}^{\infty}(\mathbb{B}^n), \int_{\mathbb{B}^n} |\nabla_{h} u|^n \, {\rm d}v_h  \leq 1} \int_{\mathbb{B}^n} 
\left (e^{\alpha_n u^{\frac{n}{n-1}}} - \sum_{j=0}^{n-2}\frac{1}{j !}\alpha_n^j|u|^{\frac{jn}{n-1}} \right) \, {\rm d}v_{h} < \infty, 
 \end{equation}
where $\alpha_n = n\omega_{n-1}^{\frac{1}{n-1}},~\omega_{n-1}$ is the $(n-1)$-dimensional measure of the unit sphere $\mathbb{S}^{n-1}.$
\end{theorem}
The above inequality \eqref{mtr} is sharp, in the sense that the ``critical'' constant $\alpha_n$ cannot be improved. See also
 \cite{LM} for Moser-Trudinger inequality in the higher dimensional hyperbolic space.

\medskip 

In \cite{GK}, the authors studied the existence of steady states for the equation \eqref{eqn:main}, with Moser-Trudinger type nonlinearity in dimension $2$. They obtained several existence 
results using concentration type arguments. In particular, there exists  non-negative and radially decreasing (and hence bounded) steady states. 
Drawing primary motivation from \cite{GK}, it is natural to ask whether one can study existence and blow-up of solutions for the heat equation with Moser-Trudinger type 
nonlinearity. There are quite a few articles closely related to these questions. Their primary interest lies in  more subtle question of local existence when the initial data is not bounded.
The case of power type nonlinearity, which is independent of $t$ i.e., $f(s, t) = s|s|^{p-1}$  has been extensively 
studied starting from the seminal papers of F. Weissler \cite{W1, W2, Weissler},  Brezis-Cazanave \cite{BC} and Giga \cite{GI}. They proved local existence of solutions for singular
 initial data $u_0 \in L^r(\mathbb{R}^n),$ for some $1\leq r< \infty.$ The {\rm non-uniqueness} of the solutions  have also been studied in \cite{NIS} for bounded domain 
and in \cite{TE} for the whole Euclidean space.
\medskip

 For exponential nonlinearity, Ioku in \cite{NI} proved existence of {\emph mild solutions} for 
 nonlinearity $f(s, t) \sim e^{s^2}$  when the initial data $u_0 \in \mbox{exp} L^2(\mathbb{R}^n)$. In \cite{NBE} the authors exhibited non-existence of 
 non-negative local solutions for exponential nonlinearity with Moser-Trudinger growth and for a particular choice of initial data. In particular they considered the following  nonlinearity given by 
 
 \begin{align}\label{nonlinearity21}
f(s) :=
\left\{\begin{array}{ll}
  \dfrac{1}{|s|^3}\, e^{s^2}, & \mbox{for}  \ |s| >  R \\  
   \kappa s^2,  & \mbox{if} \  |s| \leq R,
\end{array}\right.
\end{align}
 with appropriate $\kappa$ and $R$ such that $f$ is $C^1(\mathbb{R}),$ increasing on $[0, \infty)$ and convex on $\mathbb{R}.$ 
 
 \medskip
 
 One should notice that \eqref{eqn:main} is different from the questions that were treated in \cite{BC, GI, NI, NBE}.  As far as our knowledge, there is no literature in the study of \eqref{eqn:main}, when the nonlinearity  depends on time, and $s$ 
 exponentially growing. Moreover, we are interested in the (global) solutions for initial data in $L^{\infty}(\hn)$ and which exhibit Fujita phenomenon for certain range of parameters. 
 
 \medskip
 
  \noindent $\bullet$ \bf General Cartan-Hadamard manifolds\rm. 
  In addition to the hyperbolic space, following  the ideas of F.~Punzo in \cite{PF1}, we can extend the analogous results in the case of a Cartan-Hadamard manifold whose sectional curvature is bounded by a negative constant. 
    It is important to note that, except possibly at the borderline case, \it many \rm of the results of this article continue to hold true for Cartan-Hadamard manifolds with a pole, under the curvature \it bound \rm $K_R\le-c$, where $K_R$ being the sectional curvature
     in the radial direction (see Theorem~\ref{thm3} and Theorem~\ref{thm4}). We have first stated them in the special case of ${\mathbb H}^n$ for greater readability and because the results are optimal. 
   It is important to stress that such results will be shown to be valid even if the curvature is \it strictly negative and decays \rm at infinity. More precisely it can be allowed to tends to $-\infty$
  with rate $\sim - \hat C(1 + d(x, x_0)^{\gamma}),$ as the distance from a given pole tends to infinity, where  $\hat C >0, \gamma \geq 0,$ and $x_0$ being a pole. For more details we refer Section~\ref{manifolds}.

 \medskip
 
  \noindent $\bullet$ \bf Further results on general manifolds and comparison to the literature\rm. 
   Recently, Grillo, Meglioli and Punzo \cite{GMeP} considered a similar question on a general manifold namely, is it always necessary to amplify the non-linearity in $t$ to observe Fujita phenomena ? Their main result shows this is not always the case. In fact, they considered a time independent sources and the underlying manifold is assumed to be complete, non-compact, stochastically complete, with positive bottom of $L^2$ spectrum $\lambda_1(M)$ and satisfies Faber-Krahn inequality. There are several equivalent criterion for the Faber-Krahn inequality to hold \cite{Gri}. For example, the validity of Sobolev inequality or the point wise on-diagonal bound on the heat kernel $\mathcal{K}(t,x,x) \leq Ct^{-\frac{n}{2}},$ for all $x \in M, t$ large implies that Faber-Krahn inequality  
 \begin{align*}
 \lambda_1(\mathcal{O}) \geq \frac{c}{( \mbox{vol}(\mathcal{O}))^{\frac{n}{2}}},  \quad \mbox{for \ all \ open \ set} \ \mathcal{O} \subset M
 \end{align*}
holds.
 
 The above assumptions on the manifold allow one to get a suitable bound on the solutions of  heat equation, more precisely, on the heat kernel, which is required to distinguish between the global existence and finite time blow-up of solutions. Under the above-mentioned assumptions on the manifold, the authors showed that for a time-independent non-linearity $f(u)$ satisfying $f(0) = 0,$ convex, increasing and $1/f$ is integrable at infinity, the Fujita phenomena holds.  The non-linearities include $f(u) = e^{\beta u} -1$ and $f(u) = \lambda u,$ if $u \in [0,1]; f(u) = \lambda u^p$ for $u >1$ and for some $p >1$ as special cases. The authors showed that if $f^{\prime}(0) > \lambda_1(M)$ then all solutions blow up in finite time. On the other hand, if $f(u) \leq \lambda u$ for $u \sim 0$ for some $\lambda \leq \lambda_1(M)$ then a global solution exists. We would like to point out that the non-linearities we consider are quite different than that of \cite{GMeP}. The time independent part of our non-linearities $f(u,t) = h(t)g(u)$ satisfies $g^{\prime}(0) = 0,$ and yet we show that aiding of suitably time-dependent weight results in observable Fujita phenomena. 
Moreover, a closer look at the results of \cite{GMeP} indicates that the results concerning type-I non-linearity of this article can be extended to manifolds satisfying Faber-Krahn inequality. We would like to express our sincere thanks to Prof. Grillo for this remark.

\medskip

The paper is organized as follows: in Section~\ref{pre}, we introduce some of the notations and state our main results on the hyperbolic space, i.e., existence and blow-up of solutions. Furthermore, 
we provide definitions of different notion of solutions (classical, weak and mild solutions) and discuss their equivalence. In Section \ref{manifolds}, we state our geometric preliminaries, assumptions 
and our results on general Cartan-Hadamard manifolds.
 Section~\ref{local} is devoted to the proofs of local existence of solutions and Theorem~\ref{thm1}.
  While Section~\ref{mainthm2} contains the proofs of the results stated in Theorem~\ref{thm2} and Section~\ref{mainthm3and4} contains the proof of  Theorem~\ref{thm3} and Theorem~\ref{thm4}.
  Finally in Section~\ref{concluding remarks},  in the concluding remark we pose an open question in the limiting case for dimension $n =3.$

%%%%%%%%%%%%%%%%%%%%%%%%%%%%%%%%%%%%%%%%%%%%%%%%%%%%%%%%%%%%%%%%%%%%%%%%%%%%%%%%%%%%%%%%%%%%%%%%%%%%%%%%%%%%%%%%%%%%%%%%%%%

\section{Preliminaries and Statement of Main results}\label{pre}
%of the some generalizations of our main results to Cartan-Hadamard manifolds.

 In this section, we will introduce some of the notations and definitions used in this
 paper and statement of main results.

\subsection{Basic notions of the Hyperbolic space} The hyperbolic $n$-space is a $n$-dimensional complete, non compact Riemannian manifold having constant sectional curvature equal to $-1.$ 
There are several models for the hyperbolic  $n$-space. In this article we will only work with the ball model: the unit ball $\mathbb{B}^n := \{x \in \mathbb{R}^n: |x|^2<1\}$ given with the Riemannian metric
\begin{align*}
{\rm d}s^2 = \left(\frac{2}{1-|x|^2}\right)^2 \, {\rm d}x^2
\end{align*}
constitute the ball model for the hyperbolic $n$-space, where ${\rm d}x$ is the standard Euclidean metric and $|x|^2 = \sum_{i=1}^nx_i^2$ is the 
usual Euclidean length. The Laplace-Beltrami operator and hyperbolic volume element is given by 
\begin{align*}
\Delta_{\hn} = \left(\frac{1-|x|^2}{2}\right)^2\Delta + (n-2)\left(\frac{1-|x|^2}{2}\right)x \cdot \nabla,
\end{align*}
$ {\rm d}v_{\hn} = \left(\frac{2}{1-|x|^2}\right)^2 \, {\rm d}x,$ where $ {\rm d}x$ is the Lebesgue measure. The hyperbolic distance between two points $x,y \in \mathbb{H}^n$ will be denoted by $ \mbox{d}(x,y).$ 
The distance from a fixed point $x$ to the origin can be expressed by
\begin{align*}
\mbox{d}(x,0) = \ln \left(\frac{1+|x|}{1-|x|}\right).
\end{align*}
The geodesic ball of radius $r$ centered at the origin is given by $B_{\hn}(0,r) := \{x \in \hn \ | \ \mbox{d}(x,0)<r\}.$ For details on the hyperbolic space we refer to \cite{RAT}.

\medskip 
Let $\mathcal{K}(x,y,t)$ be the (positive minimal) heat kernel for $- \Delta_{\hn}.$ We recall  from \cite{Davies} that  $\mathcal{K}$ satisfies the 
following estimates: for $n \geq 2,$ there exist
some positive constants $A_n$ and $B_n$ such that
\begin{equation}\label{heat-estimates}
A_n h_n(x, y,t) \leq \mathcal{K}(x, y,t) \leq B_n h_n(x, y,t) \quad \mbox{for all} \ t > 0 \ \mbox{and} \ x, y \in \hn,
\end{equation}
where $h_n (x, y,t)$ is given by
\begin{equation}\label{hN}
h_n(x, y,t) := h_n(r,t) = (4 \pi t)^{-\frac{n}{2}} e^{- \frac{(n-1)^2 t}{4} - \frac{(n-1) r}{2} - \frac{r^2}{4t}} ( 1 + r + t)^{\frac{(n-3)}{2}}(1 + r),
\end{equation}
where $r = \mbox{d}(x, y).$

\medskip

It is known that for the hyperbolic space there exists a classical bounded positive solution to the eigenvalue problem
\begin{align*}
-\Delta_{\hn} \phi_1 = \lambda_1 \phi_1, \ \mbox{in} \ \hn.
\end{align*}
The function $\phi_1$ may not be in $L^{2}(\hn),$ which is why it is called a \emph{generalized eigen function}. It is well  
 known from \cite[Section 3.3]{BGGV} that, the corresponding \emph{generalized eigen function} satisfies the following estimate: 
\begin{equation}\label{bound}
\alpha (1 + r)e^{-\frac{(N-1)}{2}r}  \leq \phi_1(x) \leq \beta (1 + r)e^{-\frac{(N-1)}{2}r} \quad \text{for all } r > 0
\end{equation}
and for some $\alpha,\beta>0$ and $r := d(x, x_0),$ where $x_0$ being a  pole on the hyperbolic space. 

 \medskip
 
Our main theorems are stated below.
 
\begin{theorem}\label{thm1}
Consider equation \eqref{eqn:main} with $f$ is of the form \eqref{nonlinearity1} and $\beta>0$ fixed.  

\begin{enumerate}

\item[(a)] If $\mu \leq  p\lambda_1,$ then there exists a constant $C>0$ such that for any non-negative 
 $u_{0}\in L^{\infty}(\mathbb{H}^{n})$ satisfying $u_{0}\leq C \phi_1,$
 there exists a global in time solution to \eqref{eqn:main} corresponding to the initial data $u_{0}.$
 
 \medskip 
 
 \item[(b)] If $\mu > p\lambda_1,$ then all non-negative nontrivial solutions to \eqref{eqn:main} blows-up in finite time.
\end{enumerate}

\end{theorem}

\medskip 

\begin{remark}
As we are looking for global existence to \eqref{eqn:main} for small initial data $\|u_{0}\|_{L^{\infty}}$, and since $u\big( e^{\beta u^{p}}-1\big) \sim\beta u^{p+1}$ for $u$ small, 
it is expected that one gets global existence vs blow-up results for
  \eqref{eqn:main} similar to  \eqref{eqn:2} with $p_{F}^{*}(\mu)=p+1$. In  the above Theorem~\ref{thm1},  we answer in more detail this question by proving 
   existence of solutions and blow-up results in the spirit of Fujita \cite{Fujita} and on the approach exploited by Meier in  \cite{Meier}.

\end{remark}

\medskip

It follows from the the works of \cite{Bandle} and the previous section that the power nonlinearity in time $h(t) = t^q$ is 
subcritical for either the power nonlinearity $s^p$ or the exponential nonlinearity $s(e^{\beta s^p} - 1)$ in the hyperbolic space. 
 Next we shall consider the question of what is an appropriate nonlinearity in $s$ so that $h(t) = t^q$ is indeed becomes
  critical for some $q.$ In other words, find an appropriate nonlinearity $g(s)$  so that the Fujita phenomena exhibits in the
   problem. This led us to the second theorem. 

\medskip

We will see that the  role of dimension comes into play when the  global existence  and finite time blow-up are concerned. In fact, 
we will see that if $q > \alpha,$  all the solutions to \eqref{eqn:main} with $f$ as stated in \eqref{nonlinearity2} blows-up in finite 
time while for $q = \alpha$ blow-up occurs in dimension $2$ and global existence in dimension $n \geq 4.$ Moreover, these exponents are \emph{optimal}. The phenomenon is still not clear in dimension $3$ which 
remains an \emph{open question}. See concluding remark at the end of this article.

To be precise, we prove the following

\medskip

\begin{theorem}\label{thm2}

Consider equation \eqref{eqn:main} with $f$ is of the form \eqref{nonlinearity2}. For a given  $q> 0$ and  $\alpha >0,$ there hold
\begin{enumerate}

\item[(a)]  If $q < \alpha,$ then there exists a constant $C$ such that for any non-negative data $u_0 \in L^{\infty}(\hn)$ satisfying $u_0 \leq C\phi_1$, 
\eqref{eqn:main} admits a global solution corresponding to the initial data $u_0.$

\medskip 

\item[(b)] If $q > \alpha,$ all non-negative solutions blow-up in finite time.

\medskip 

\item[(c)]  If $q = \alpha$ then
\begin{itemize}
\item[(i)] in dimension $n = 2,$  all non-negative solutions blow-up in finite time.

\medskip 
\item[(ii)] for dimension $n \geq 4,$ then there exists a constant $C$ such that for any non-negative data $u_0 \in L^{\infty}(\hn)$ satisfying $u_0 \leq C\phi_1$, \eqref{eqn:main} admits a global solution corresponding to the initial data $u_0.$ 
\end{itemize}

\end{enumerate}
\end{theorem}

\medskip 

\begin{remark}
It follows from the proof of Theorem \ref{thm2} that the assumption $h(s) \geq \kappa s^2$ in \eqref{nonlinearity2} can be relaxed by simply assuming $\frac{1}{h} \in L^1(\frac{1}{2}, \infty).$
This suggests that the effect of Fujita phenomenon arises from nonlinearity defined near $s \sim 0,$ which is not so evident from all the previously chosen nonlinearity in the literature. 
\end{remark}

\medskip 

\subsection{Classical, mild and weak solutions} In this subsection we recall three different notions of solution to the problem \eqref{eqn:main}.  Set $0< T <\infty.$

\begin{definition}\label{defi1}
A function $u \in C^{2, 1}(\hn \times (0, T]) \cap C(\overline{\hn \times (0, T]})$ is called a \emph{classical solution} of the
problem \eqref{eqn:main} in $[0, T]$ if 
\begin{align}\label{eq:defi1}
\left\{\begin{array}{ll}
\partial_{t}u=\Delta_{\mathbb{H}^{n}} u+ f(u, t) &\hbox{ in }~ \mathbb{H}^{n}\times (0, T],\\
\\
\quad u(x, 0) =u_{0}(x)  &\hbox{ in }~ \mathbb{H}^{n}.
\end{array}\right.
\end{align}
\end{definition}

 \begin{definition}\label{defi2}
A function $u \in C(\hn \times [0, T]) \cap L^{\infty}(\hn \times (0,T))$ is called a \emph{mild solution} of  theproblem \eqref{eqn:main} if the identity
$$
u(x, t) = \int_{\hn} \mathcal{K}(t, x, y) u_0(y) \, {\rm d}v_{\hn} + \int_{0}^{t} \int_{\hn} \mathcal{K}(t-s, x, y) f(u(y,s), s) \, {\rm d}v_{\hn} \, {\rm d}s,
$$
holds for every $t \in [0, T]$, where $\mathcal{K}$ is the positive minimal heat kernel for $-\Delta_{\hn}.$ 
\end{definition}

\begin{definition}\label{defi3}
A function $u \in C(\hn \times [0, T]) \cap L^{\infty}(\hn \times (0,T))$ is called a \emph{continuous weak solution} of the problem \eqref{eqn:main} if for any 
$T_1 \in (0, T]$ 

$$
-  \int_{0}^{T_1}\int_{\hn} u \{ \Delta_{\hn} \psi + \psi_t \} \, {\rm d}v_{\hn} {\rm d}t = \int_{\hn} u_0 \, \psi(., 0) \, {\rm d}v_{\hn} + 
 \int_{0}^{T_1} \int_{\hn}f(u, t) \, \psi \, {\rm d}v_{\hn} \, {\rm d}t,
$$
holds for every $\psi \in C_c^{2, 1}(\hn \times [0, T_1)).$ 
\end{definition}

In this sequel we will be using the following definition of blow-up of solutions:

\begin{definition}
Let $u$ be a continuous weak solution of the problem \eqref{eqn:main} for $t \in [0, T),$ where $T > 0$ is the maximal time of existence of $u.$ If $T$ is finite we have 
$$
\lim_{t \rightarrow T^-} \, ||u(., t)||_{\infty} = \infty.
$$
 $T$ is called the blow-up time of the solution and $u$ is said to blow up in finite time.  If there is no finite time blow-up then the solution exists for all time $t>0$ and we call it a  \it global solution \rm.
\end{definition}

We conclude this section by recalling the following well known results. For the nonlinearities considered in this article, a mild solution is a classical solution and hence a weak solution. It turns out that that in the class of bounded solutions all notions are equivalent.

\begin{theorem} \label{weak=mild}
Let $u$ be a continuous bounded weak solution of \eqref{eq:defi1} in $Q_{T}.$ Then $u$ is a mild solution to \eqref{eq:defi1} in $Q_T.$
\end{theorem}

The proof of Theorem \ref{weak=mild} can be found in \cite{Bandle} relying on the following maximum principle, which will also be used to prove the uniqueness of local 
weak solutions.

\begin{lemma}\label{uniqueness}
Let $v$ be a continuous weak sub-solution of the heat equation i.e
\begin{align*}
\int\int_{Q_T} v(x,t) (\partial_t\psi(x,t) + \Delta_{\hn} \psi(x,t)) \leq 0, ~\forall \psi \in C_c^{\infty}(\hn \times (0,T)), \psi \geq 0
\end{align*}
which satisfies the growth condition 
\begin{align*}
v(x,t) \leq Ae^{cd(x,0)^2}, \ for \ all \ t \in (0,T)
\end{align*}
for some constants $A,c>0$ and assume that $v(\cdot,0) \leq 0.$ Then $v \leq 0$ on $Q_T.$
\end{lemma}

\medskip

%%%%%%%%%%%%%%%%%%%%%%%%%%%%%%%%%%%%%%%%%%%%%%%%%%%%%%%%%%%%%%%%%%%%%%%%%%%%%%%%%%%%%%%%%%%%%%%%%%%%%%%%%%%%%%%%%%%%%%%

\section{Results on Cartan-Hadamard manifolds}\label{manifolds}

Most of the results of this article can be extended to Cartan-Hadamard manifolds of dimension $\geq 2.$ A 
Cartan-Hadamard manifold $(M,g)$ is a geodesically complete, non-compact, simply connected Riemannian manifold with non-positive sectional curvature.
In the present section, we state a generalization of  Theorems~\ref{thm1}  and \ref{thm2} to more general manifolds
under suitable curvature assumptions. Denote by $K_R$ the sectional curvature in the radial direction of a Riemannian manifold with a pole $x_0$. Throughout this article, we assume the bound
\begin{equation}\label{curv}
K_R(x)\le -G(r(x))\le0\qquad  \forall x\in M
\end{equation}
holds, where $G$ is a given function and $r(x)= {\rm d}(x,x_0)$. In particular, we are assuming that $M$ is Cartan-Hadamard. We also define $\psi$ to be the solution to the Cauchy problem
\begin{equation}\label{curvmod}
\begin{cases}\psi''(r)-G(r)\psi(r)=0 \qquad r>0,
\\ \psi(0)=0,\ \psi'(0)=1.\end{cases}
\end{equation}
Clearly, by the sign assumption on $G$, $\psi$ is positive convex function, and in particular, by the initial condition we have $\psi(r)\ge r$ for all $r\ge0$. 

%Then, to each function $\psi\in C^{2}([0,\infty)),$ strictly positive on $(0, \infty)$ and such that $\psi(0) = \psi^{\prime \prime}(0)= 0$, $\psi^{\prime}(0) = 1$, we associate a family of manifolds defined as follows:
%   \begin{equation}\label{condition2}
% \mathcal M_\psi=\{ \text{$M$ Riemannian manifold with pole $x_0$}:\,    Cut \{x_0 \} = \phi \text{ and } K_{R}(x) \leq - \frac{\psi^{\prime \prime}}{\psi}   \quad \forall\, x\in M\}\,.
%   \end{equation}
To construct a positive super-solution one can use the well-known strategy of constructing barriers using Hessian comparison and equations posed on the \emph{Riemannian model} $M_{\psi}$ associated to $\psi$ 
constructed above. Namely, we consider the $N$-dimensional Riemannian manifold $M_{\psi}$ admitting a pole  $x_0$,  whose metric is given in spherical coordinates by
 \begin{equation}\label{meetric}
 {\rm d}s^2 = {\rm d}r^2 + \psi^2(r) \, {\rm d}\omega^2,
 \end{equation}
 where $ {\rm d}\omega^2$ is the standard metric on the sphere $\mathbb{S}^{N-1}$. The coordinate $r$ represents the Riemannian distance from the pole $x_0,$ see e.g. \cite{GW} for further details. For Riemannian 
 models the curvature condition in \eqref{curv} holds with an equality. Clearly, for $\psi(r)=r$ one has $M_{\psi}=\mathbb R^n$, while for $\psi(r)=\sinh r$ one has $M_{\psi}=\hn$.
 
 \medskip 
 
 Let $(M,g)$ be a Cartan-Hadamard manifold. On Cartan-Hadamard manifold $M,$ for any point $x_0 \in M,$ the cut locus of $x_0,$ $Cut(x_0)$ is empty. Thus $M$ is a manifold with pole.
The Riemannian metric $g$ in $M$ in the polar coordinates takes the form
\begin{align*}
{\rm d}s^2 = {\rm d} r^2 + a_{i,j}(r,\theta) {\rm d}\theta_{i}{\rm d}\theta_{j},
\end{align*}
where $(\theta_{1},\ldots, \theta_{n-1})$ are coordinates on $\mathbb{S}^{n-1}$ and $((a_{i,j}))_{i,j= 1, \ldots, n}$ is a positive
definite Matrix.

Let $a: = \det(a_{i,j}),$ $B(x_0,\rho) = \{ x = (r, \theta): r  < \rho \}. $ Then in $M$ we have
\begin{align*}
 \Delta_{M} = \frac{1}{\sqrt{a}} \frac{\partial}{\partial r}\left( \sqrt{a} \frac{\partial}{\partial r} \right)
+ \Delta_{\partial B(o,r)} = \frac{\partial^2}{\partial r^2} + m(r, \theta) \frac{\partial}{\partial r}
+ \Delta_{\partial B(o,r)},
\end{align*}
where $\Delta_{\partial B(x_0, r)}$ is the Laplace-Beltrami operator on the geodesic sphere $\partial B(x_0,r)$ and $m(r, \theta)$
is a smooth function on $(0, \infty) \times \mathbb{S}^{n-1}$ which represents the mean curvature of $\partial B(x_0,r)$ in the
radial direction.

 \medskip 
 
 Let us recall a Lemma. 
 
 \begin{lemma}\label{comp}\cite{GW}
Let $M$ be a Cartan-Hadamard manifold with a pole at $x_0$ satisfying the assumption
\begin{align*}
K_{R}(x) \leq - \frac{\psi^{\prime \prime}}{\psi}\quad \forall\, x\in M.
\end{align*}
Then
 \[
 m(r, \theta) \geq (n-1) \frac{\psi^{\prime}(r)}{\psi(r)}\quad \text{for all } r>0 \text{ and } \theta\in\mathbb{S}^{n-1}.
 \]
 \end{lemma}

\medskip 

  \noindent $\bullet$ {\bf Geometric Assumptions.}  We further assume following uniform curvature bound
  
\begin{enumerate}
\item[{\bf(G1)}] there exists a constant $\kappa >0$  such that for any $x \in M$ and for any plane $\pi \subset T_xM$
there holds $K_{\pi}(x) \leq -\kappa^2,$ where $K_{\pi}$ is the sectional curvature of the plane $\pi.$

\medskip 

\item[{\bf(G2)}] Let $K_{R}(x)$ denote the radial sectional curvature at the point $x.$ We assume 
\begin{align} \label{curvature bound}
K_R(x) \leq -\hat C (1 + d(x_0, \, x)^{\gamma}) \ \mbox{for some $o \in M, \hat C >0$ and $\gamma \geq 0.$}
\end{align}

\medskip
\end{enumerate}

%%%%%%
\begin{remark}
For the hyperbolic space $\gamma = 0,$ the above estimate with $\gamma > 0$ are satisfied if, for example, we take the model metric of the form 
\begin{align*}
\psi(\rho) = e^{f(\rho)} \ and \ \psi(\rho) = O(\rho^{1 + \frac{\gamma}{2}}) \ as \ r \rightarrow \infty\  \ if \ \gamma >0.
\end{align*}
We refer to \cite[Section~2.3]{GMV} for details on the geometry of such manifolds. Recently, in the context of nonlinear evolution equation of Porous Medium type 
on Cartan-Hadamard manifolds, such assumptions on radial sectional curvature arises in the seminal works of  Grillo-Muratori-Punzo (see \cite{GMP1, GMP2} and references therein).
\end{remark}

%%%%%%
We consider only two special subclass of the problem. If $\gamma >0,$ then there exists a positive bounded super-solution to $-\Delta_M \phi  = \lambda_1(M) \phi$ which follows from the work of F.~Punzo \cite{PF1}
and we can get sharp bounds on the exponent such that Fujita phenomena holds. On the other hand if we allow $\gamma \geq 0,$ then we can prove some partial results stated below. The main idea rests on the 
existence of positive bounded super-solution. In \cite{PF1}, it was shown using Lemma~\ref{comp} that {\it generalised eigenvalue} problem admits positive bounded super-solution under the above mentioned assumptions on the manifold. For the hyperbolic space there exists a bounded  {\it ground state} which suffices the purpose. If we consider, $\psi(r) = \frac{1}{\kappa} \sinh (\kappa r)$ then the model manifold has constant sectional curvature $-\kappa^2,$ 
the bottom of the spectrum is given by $\frac{(n-1)^2}{4}\kappa^2$ and there exists a bounded ground state. In general, the behaviour of the {\it ground state} at infinity is not known. In order to prove the existence of a global solution we need 
 a comparison principle of weak solutions to \eqref{eqn:main} and a positive bounded weak super solution (see \cite{PF1}) to the {\it eigenvalue problem} :
  \begin{align} \label{evpM}
-\Delta_M \phi = \lambda_1(M) \phi \ \ in \ M.
\end{align}
Thanks to the above geometric assumptions, such properties hold true (see \cite{PF1}). 
On the other hand the proof of finite time blow-up of solutions depend on the following properties of the heat kernel $\mathcal{K}_M.$

\medskip 

\begin{enumerate}
\item[{\bf(H1)}] The heat kernel satisfies the upper bound
\begin{align*}
\mathcal{K}_M(x,y,t) \leq \frac{C}{(\min\{t,T\})^{\frac{n}{2}}}\left(1 + \frac{d(x,y)^2}{t}\right)^{\frac{n}{2}}e^{-\frac{d(x,y)^2}{4t} - \lambda_1(M)(t-T)_+}
\end{align*}
for all $x,y \in M$ and $t,T \in (0,\infty).$ Here $d(x,y)$ denotes the geodesic distance and $C$ is a positive constant. 

\medskip

\item[{\bf(H2)}] Moreover, for large time we have the following lower bound on $\mathcal{K}_M(x,y,t)$ or more precisely, the following asymptotic estimate holds 
\begin{align*}
\lim_{t \rightarrow \infty} \frac{\ln \mathcal{K}_M(x,y,t)}{t} = -\lambda_1(M), 
\end{align*}
 locally  uniformly  on  $M\times M.$
 \end{enumerate}
\medskip 

In order to state our main results on the Cartan-Hadamard manifold, let us introduce the following notation
\begin{align} \label{lstarm}
\lambda_*(M) = 
\begin{cases}
\lambda_1(M), \ \ \ \ \ \mbox{if} \ \gamma >0 \\
\frac{(n-1)^2}{4}\kappa^2, \ \ \ \mbox{if only} \ \gamma \geq 0 \ \mbox{ is  assumed},
\end{cases}
\end{align}
where $\gamma$ is given by \eqref{curvature bound} representing the decay in the radial sectional curvature.  Note that when $M$ is a Cartan-Hadamard manifold and satisfies $(G1),$
 then it follows from \cite{MP} that 
 
 $$
 \lambda_1(M) \geq \frac{(n-1)^2}{4} \kappa^2,
 $$
and hence $\lambda_{*}(M) \leq \lambda_1(M).$

\medskip

Now by exploiting \cite[Proposition 3.1 and 4.1]{PF1}, there exists a continuous bounded weak super-solution to the eigenvalue problem \eqref{evpM} with $\lambda_1(M)$ replaced by $\lambda_*(M)$, which we denote by $\phi_M.$

 For Type-II nonlinearity we have the following results: 

\begin{theorem}\label{thm3}
Let $(M,g)$ be a Cartan-Hadamard manifold satisfying the hypotheses ${\bf (G1)}$ and ${\bf(G2)}.$ Consider the equation \eqref{eqn:main} in $M$ with $f$ is of the form \eqref{nonlinearity1}.
\begin{enumerate}

\item[(a)] If $\mu <  p\lambda_*(M),$ then there exists a constant $C>0$ such that for any non-negative $u_{0}\in C(\hn) \cap L^{\infty}(\mathbb{H}^{n})$ satisfying $u_0 \leq C \phi_M$
 there exists a global in time solution to \eqref{eqn:main} corresponding to the initial data $u_{0}.$
 
 \medskip 
 
 \item[(b)] If $\mu > p\lambda_1(M),$ then all non-negative nontrivial solutions to \eqref{eqn:main} blow-up in finite time.
\end{enumerate}

\end{theorem}

%For $\gamma \geq 0$ we have some partial result
%\begin{theorem}\label{thm4}
%Let $(M,g)$ be a Cartan-Hadamard manifold satisfying the hypotheses $(G1), (G2)$ with $\gamma \geq0.$ Consider equation \eqref{eqn:main} with $f$ is of the form \eqref{nonlinearity1} and let $\phi_M$ be a positive, 
%bounded super-solution to \eqref{evpM}.  

%\begin{enumerate}

%\item If $\alpha \leq q\frac{(n-1)^2}{4}\kappa^2,$ then there exists a constant $\tau(\beta)>0$ such that for any non-negative $u_{0}\in L^{\infty}(\mathbb{H}^{n})$ satisfying $u_0 \leq \tau(\beta) \phi_M$
 %there exists a global in time solution to \eqref{eqn:main} corresponding to the initial data $u_{0}.$
 
% \medskip 
 
% \item If $\alpha > q\lambda_1(M),$ then all nontrivial solutions to \eqref{eqn:main} blows-up in finite time.
%\end{enumerate}
%\end{theorem}

\begin{remark}
It is interesting to note that in the above theorem part $(a),$ existence of global solution is proved when $\mu <  p\lambda_*(M).$ So the theorem does not 
cover the case, when $\lambda_*(M) < \lambda_1(M).$ This is precisely because of the existence of  bounded  super-solution for the corresponding eigenvalue  problem 
is not yet known.  Hence we can not exploit our method in this case. 

\end{remark}

\medskip

On the other hand for Type-I nonlinearity we have the following results:

\begin{theorem}\label{thm4}

Let $(M,g)$ be a Cartan-Hadamard manifold satisfying the hypotheses ${\bf (G1)}$ and ${\bf(G2)}.$ Consider equation \eqref{eqn:main} in $M$ with $f$ is of the form \eqref{nonlinearity2}.
\begin{enumerate}

\item[(a)]  If $q < \alpha,$ then there exists a constant $C$ such that for any non-negative initial data $u_0 \in C(\hn) \cap L^{\infty}(\hn)$
satisfying $u_0 \leq C\phi_M$, \eqref{eqn:main} admits a global solution. 
\medskip 

\item[(b)] If $q > \alpha,$ then all non-negative, nontrivial solutions blow-up in finite time.

\medskip 

\item[(c)]  If $q = \alpha$ and $n > 1 + \frac{2\sqrt{2}}{\kappa},$ then \eqref{eqn:main} admits a global classical solution for non-negative initial data satisfying $u_0 \leq C\phi_M$.

\end{enumerate}
\end{theorem}

%%%%%%%%%%%%%%%%%%%%%%%%%%%%%%%%%%%%%%%%%%%%%%%%%%%%%%%%%%%%%%%%%%%%%%%%%%%%%%%%%%%%%%%%%%%%%%%%%%%%%%%%%%%%%%%%

\section{Local Existence and Proof of Theorem~\ref{thm1}}\label{local}

\subsection{Local existence}
We prove local existence of solution to problem~\eqref{eqn:main}. The proof relies on the sub-super solution technique. Let $\{ D_k \}_{k =1}^{\infty}$ be a compact exhaustion of $\hn,$ i.e. a sequence of smooth, relatively compact domains 
in $\hn$ such that $D_1 \neq \phi,$ $ \overline{D_k} \subset D_{k +1}$ for every $k \in \mathbb{N},$ $\cup_{k =1}^{\infty} D_k = \hn.$

\medskip 

\begin{theorem}
Given any non-negative data $u_0 \in L^{\infty}(\hn) \cap C(\hn),$ there exists $T > 0$ such that the problem \eqref{eqn:main} admits a unique solution 
$u \in L^{\infty}(\hn)$ corresponding to the initial data $u_0.$ The solution is either global, or there exists a maximal time of existence $T_{{\tiny\mbox{max}}}< \infty,$ such that 
\begin{align*}
\lim_{t \uparrow T_{{\tiny\mbox{max}}}} ||u(\cdot, t)||_{L^{\infty}(\mathbb{H}^n)} = \infty.
\end{align*}
\end{theorem}

\begin{proof}

The proof is based on constructing a suitable sub-super solution. Let us first consider the case, when the nonlinearity is of  type~\eqref{nonlinearity1}. 
Set $h(t) := e^{\mu t},$ for $t \geq 0.$  We seek a non-negative super-solution $\overline{U}$ of the form
$$
\overline{U}(x, t) = Ca(t), \quad C > 0.
$$
Then $\overline{U}$ satisfies
\begin{align*}
a^{\prime}(t) & = h(t) \, a(t) \, (e^{ \beta (Ca(t))^p} - 1). 
\end{align*}
Now imposing the initial condition on $\overline{U}$ so that it becomes a  super-solution, we require
$$
Ca(0) \geq u_0(x), \quad \mbox{for}  \ x \in \hn.
$$ 
Therefore, we  can choose $a(0) = 1$ and $C = ||u_0||_{L^\infty(\hn)}.$ Moreover the following O.D.E
\begin{align*}
\left\{\begin{array}{ll}
  a^{\prime}(t)  = h(t) \, a(t) \, (e^{\beta(Ca(t))^p} - 1)\\  
     a(0) = 1,
\end{array}\right.
\end{align*}
with $C = ||u_0||_{L^\infty(\hn)},$ admits a unique local solution, i.e., there exists $T_0 \leq \infty$ such that the solution $a(t)$ exists for all $t \in [0, T_0).$
Define $\overline{U}(x, t) : = ||u_0||_{L^\infty(\hn)} a(t),$ where $a(t)$ solves the above problem (locally). Then clearly $\overline{U}$ is a super-solution 
of equation  \eqref{eqn:main}. In view of our non-negativity assumption on $u_0,$ clearly  $\underline{U}(x, t) \equiv 0$ is sub-solution of  problem 
\eqref{eqn:main}. Following \cite{Bandle}, one can see immediately that $\overline{U}$ and $\underline{U}$ are super and sub-solutions to the following limiting problems:
\begin{align*}
\left\{\begin{array}{ll}
  u_t \, = \, \Delta_{\hn} + e^{\mu t}\, u\, (e^{\beta \, u^p} - 1) \quad \mbox{in} \ D_k \times (0, T_0) \\
  u \, = \, 0 \quad \quad \quad \quad \quad \quad \quad \quad  \quad \quad \quad  \mbox{on} \  \partial{D_k} \times (0, T_0)   \\
  u \, = \, u_0 \quad \quad \quad \quad \quad \quad \quad \quad  \quad \quad \quad  \mbox{on} \  {D_k} \times \{ 0 \},
\end{array}\right.
\end{align*}

where $D_k$ is an exhaustion of $\hn.$ By standard regularity and monotonicity arguments, there exists a function $u_k \in C(D_k \times [0, T_0)),$  $0 \leq u_k \leq \overline{U}$ which satisfies 
the above equation. Now by using standard iteration technique we obtain a solution $U$ for \eqref{eqn:main}. Uniqueness follows using lemma~\ref{uniqueness} and \cite[Theorem~3.1]{Bandle}
with the necessary modifications. 

If $T_{{\tiny\mbox{max}}}+||u(\cdot, T_{{\tiny\mbox{max}}})||_{L^{\infty}(\mathbb{H}^n)} < \infty$ then we can redo the process with initial data $u(\cdot, T_{{\tiny\mbox{max}}})$ and initial time $T_{{\tiny\mbox{max}}}$ and get a local solution. By uniqueness of the solutions, we get a contradiction on the maximal time of existence.

\medskip 

A similar approach will provide us a solution, possibly local, for nonlinearity of  type \eqref{nonlinearity2}.

\end{proof}

\subsection{Proof of Theorem~\ref{thm1}(a)}
We break the proof into two steps. We first note that the global existence to the Cauchy problem \eqref{eqn:main} with small initial data in $L^{\infty}(\mathbb{H}^{n})$  for some $\beta_{0}>0$, implies the same for any $\beta>0$. 
 
\begin{step}
Suppose for some $\beta_{0}>0$ there exists a global in time solution $v$ to \eqref{eqn:main} with initial data $v_0.$ Then for any $\beta>0$ define $u = \left(\frac{\beta_{0}}{\beta}\right)^{\frac{1}{p}} ~v.$ Then
%$\|v_{0}\|_{L^{\infty}}\leq \tau(\beta_{0})$, for some $ \tau(\beta_{0})>0$. Then for any $\beta>0$ there %exists $\tau(\beta)>0$ such that there exists a global in time solution to \eqref{eqn:main} with the initial data %$%\|u_{0}\|_{L^{\infty}}\leq \tau(\beta)$.
\begin{align*}
\partial_{t}u-\Delta_{\mathbb{H}^{n}} u=\left(\frac{\beta_{0}}{\beta}\right)^{\frac{1}{p}} \big(\partial_{t}v-\Delta_{\mathbb{H}^{n}} v\big)
=e^{\mu t}u\big( e^{\beta u^{p}}-1\big)~\hbox{ in }  \mathbb{H}^{n}\times (0, +\infty).  \notag\\
\end{align*}
Hence $u$ is a global in time solution to \eqref{eqn:main}  with the initial data 
$u_0 = \left(\frac{\beta_{0}}{\beta}\right)^{\frac{1}{p}} ~v_0.$
%$$\|u_{0}\|_{L^{\infty}}=\left(\frac{\beta_{0}}{\beta}\right)^{1/2}\|v_{0}\|_{L^{\infty}}\leq \left(\frac{\beta_{0}}{\beta}\right)^{1/2}\tau(\beta_{0}):=\tau(\beta).$$
\end{step}

We now produce a global in time super-solution for small $\beta$, which in turn provides a global in time solution to \eqref{eqn:main}.
\begin{step}
Let $\mu \leq p \lambda_1$. Then there exists $\beta_{0} \in (0,1)$ such that  the initial value problem \eqref{eqn:main} admits a  super-solution.
\end{step}
\begin{proof}
It follows from the works of Bandle-Pozio-Tesei \cite{Bandle} and Wang-Yin \cite{WangYin} that  there exists a global in time solution to
\begin{align}\label{eqn:3}
\left\{\begin{array}{ll}
\partial_{t}w=\Delta_{\mathbb{H}^{n}} w+e^{\mu t}w^{p+1} &\hbox{ in }~ \mathbb{H}^{n}\times (0, +\infty),\\
\\
\quad w =u_{0}\in L^{\infty}(\mathbb{H}^{n}) &\hbox{ in }~ \mathbb{H}^{n}\times \{0\},
\end{array}\right.
\end{align}
for initial data $u_0$ satisfying the assumption of Theorem \ref{thm1}(a), provided $p \geq \frac{\mu}{\lambda_1}$.  Moreover $\sup \limits_{ t>0}\|w(\cdot,t)\|_{L^{\infty}(\mathbb{H}^{n})}<+\infty$.

We choose  $\beta_{0}$ such that
\begin{align*}
0<\beta_{0}< \frac{1}{1+(\sup \limits_{ t>0}\|w(\cdot,t)\|_{L^{\infty}(\mathbb{H}^{n})})^p}.
\end{align*}
Then 
\begin{align*}
\partial_{t}w-\Delta_{\mathbb{H}^{n}} w=e^{\mu t}w^{p+1} \geq e^{\mu t} w \big( e^{\beta_0 w^{p}}-1\big),
\end{align*}
where we have used the inequality 
\begin{align*}
e^{\beta_0 s} \leq s+1 \quad \hbox{ for all }~ 0\leq s\leq \frac{1-\beta_0}{\beta_0}.
\end{align*}

So $w$ is a super-solution to \eqref{eqn:main}. Moreover $\sup \limits_{ t>0}\|w(\cdot,t)\|_{L^{\infty}(\mathbb{H}^{n})}<+\infty$.

We fix a time $T$ and as before, consider a compact exhaustion $\{D_k\}$ of $\mathbb{H}^n.$ Since $w$ 
is a super-solution which is bounded through out time, the problem \eqref{eqn:main} posed in $D_k$ with Dirichlet boundary data admits a unique solution $u_k$ with initial data $u_0.$ By standard elliptic regularity, monotonicity and Cantor diagonal argument, we can extract a sub-sequence converging to a continuous weak solution to \eqref{eqn:main} subject to the initial data $u_0,$ completing the proof of global existence. 
\end{proof}
\subsection{Proof of Theorem~\ref{thm1}(b)}
 Assume $\mu >p\lambda_1$. Recall $\mathcal{K}$ denotes the heat kernel in $\mathbb{H}^{n}$.  We proceed as in Bandle-Pozio-Tesei \cite{Bandle}.
Let $u$ be a mild solution of \eqref{eqn:main}, that is
\begin{align}\label{eqmild}
u(x,t)=&~\int \limits_{\mathbb{H}^{n}}\mathcal{K}(x,y,t) \, u_{0}(y)~{\rm d}v_{\mathbb{H}^{n}}(y) \notag\\
&  +\int \limits_{0}^{t}\int \limits_{\mathbb{H}^{n}}\mathcal{K}(x,y,t-s)e^{\mu s}u(y,s)\big( e^{\beta u^{p}(y,s)}-1\big)~{\rm d}v_{\mathbb{H}^{n}}(y) \, {\rm d}s.
\end{align}
For $u_{0}\in C(\mathbb{H}^{n})\cap L^{\infty}(\mathbb{H}^{n})$, let $T_{u_{0}}$ denote the maximal interval of existence of solution $u$ of \eqref{eqn:main}.  We set  for $T < T_{u_0}$ and $(x,t)\in \mathbb{H}^{n}\times [0,T)$
\begin{align*}
\Phi(x,t):= \int \limits_{\mathbb{H}^{n}}\mathcal{K}(x,y,T-t)u(y,t)~{\rm d}v_{\mathbb{H}^{n}}(y)
\end{align*}
Then 
\begin{align*}
\Phi(x,0)= \int \limits_{\mathbb{H}^{n}}\mathcal{K}(x,y,T)u_{0}(y)~{\rm d}v_{\mathbb{H}^{n}}(y)= \left(e^{T\Delta_{\mathbb{H}^{n}}}u_{0} \right)(x). 
\end{align*}

Multiplying \eqref{eqmild} by $\mathcal{K}(x,z,T-t)$ and integrating on $\mathbb{H}^{n}$, with respect to the $z$ variable and using semigroup property of heat kernel we obtain 
\begin{align*}
& \int \limits_{\mathbb{H}^{n}}\mathcal{K}(x,z,T-t)u(z,t)~{\rm d}v_{\mathbb{H}^{n}}(z)=~\int \limits_{\mathbb{H}^{n}}\mathcal{K}(x,y,T) \, u_{0}(y)~{\rm d}v_{\mathbb{H}^{n}}(y) \notag\\
&  +\int \limits_{0}^{t}\int \limits_{\mathbb{H}^{n}}\mathcal{K}(x,y,T-s)e^{\mu s} \, u(y,s)\big( e^{\beta u^{p}(y,s)}-1\big)~{\rm d}v_{\mathbb{H}^{n}}(y) \, {\rm d}s.
\end{align*}
Therefore,
\begin{align*}
\Phi(x,t)=\Phi(x,0)+\int \limits_{0}^{t}\int \limits_{\mathbb{H}^{n}}\mathcal{K}(x,y,T-s)e^{\mu s}u(y,s)\big( e^{\beta u^{p}(y,s)}-1\big)~{\rm d}v_{\mathbb{H}^{n}}(y) \, {\rm d}s.
\end{align*}
Since $u \mapsto u\big( e^{\beta u^{p}}-1\big)$ is convex and 
\begin{align*}
\int \limits_{\mathbb{H}^{n}}\mathcal{K}(x,y,T-s)~{\rm d}v_{\mathbb{H}^{n}}(y)=1,
\end{align*}
we get by Jensen's inequality
\begin{align*}
\int \limits_{\mathbb{H}^{n}}\mathcal{K}(x,y,T-s)~u(y,s)\big( e^{\beta u^{p}(y,s)}-1\big)~{\rm d}v_{\mathbb{H}^{n}}(y)\geq \Phi(x,s)\big( e^{\beta \Phi^{p}(x,s)}-1\big).
\end{align*}
Therefore,
\begin{align*}
\Phi(x,t)-\Phi(x,0)\geq \int \limits_{0}^{t}e^{\mu s}\Phi(x,s)\big( e^{\beta \Phi^{p}(x,s)}-1\big) \, {\rm d}s,
\end{align*}
and hence for all $0\leq t\leq T$
\begin{align*}
\partial_{t}\Phi(x,t) \geq e^{\mu t}\Phi(x,t)~\big( e^{\beta \Phi^{p}(x,t)}-1\big) \geq  \frac{\beta^{k}}{k!} \Phi^{pk+1}(x,t)e^{\mu t},
\end{align*}
which in turn implies
\begin{align*}
\frac{k!}{\beta^{k}}  \frac{\partial_{t}\Phi(x,t)}{\Phi^{pk+1}(x,t)}~ \geq e^{\mu t}, \ \ \ for \ all \ k \in \mathbb{N}. 
\end{align*}
Integrating we obtain $\frac{(k-1)!}{p\,\beta^{k}}  \frac{1}{\Phi^{pk}(x,0)} ~\geq \frac{e^{\mu t}-1}{\mu},$ for all 
$k \in \mathbb{N}.$ So we get the upper bound 
\begin{align} \label{upperbound}
1\geq  \frac{p \, \beta^{k}}{(k-1)!}\left(\frac{e^{\mu t}-1}{\mu} \right)\Phi^{pk}(x,0)\quad \hbox{ for } 0\leq t\leq T.
\end{align}
Now, we use the following estimate from Bandle-Pozio-Tesei \cite{Bandle}:
\begin{align} \label{lb}
\Phi(x,0)\geq C(x) T^{-3/2}e^{-\lambda_{1}T},
\end{align}
where $C(x)>0$. So \eqref{upperbound} becomes
\begin{align*} 
\frac{p\beta^{k}}{(k-1)!}\left(\frac{e^{\mu T}-1}{\mu} \right)T^{-\frac{3pk}{2}}e^{-pk\lambda_{1}T}C(x)^{pk} ~ \leq 1, 
\end{align*}
hence, we conclude 
\medskip
\begin{align} \label{abc123}
&\Bigg[p\left(\frac{e^{\mu T}-1}{\mu} \right)\frac{T^{\frac{3p}{2}}e^{p\lambda_{1}T}}{\beta C(x)^p} \Bigg] \frac{1}{(k+1)!}\Big(\beta C(x)^pT^{-\frac{3p}{2}}e^{-p\lambda_{1}T}\Big)^{k+1} \notag\\
&~ \leq \frac{1}{k(k+1)}, 
\end{align}
Set $\Upsilon = \beta C(x)^pT^{-\frac{3p}{2}}e^{-p\lambda_{1}T}.$ Summing over all $k \in \mathbb{N}$ we obtain from \eqref{abc123}
\begin{align*}
&\Bigg[p\left(\frac{e^{\mu T}-1}{\mu} \right)\frac{T^{\frac{3p}{2}}e^{p\lambda_{1}T}}{\beta C(x)^p} \Bigg] ~\sum \limits_{k=1}^{+\infty} \frac{\Upsilon^{k+1}}{(k+1)!} \leq \sum \limits_{k=1}^{+\infty} \frac{1}{k(k+1)} \leq~ 1,\notag\\
i.e. \ \ &\left(\frac{e^{\mu T}-1}{\mu} \right)T^{\frac{3p}{2}}e^{p\lambda_{1}T}\left(e^{\Upsilon}-\Upsilon-1 \right)  \leq ~\frac{\beta C(x)^p}{p}, \notag\\
i.e. \ \ &\left(\frac{e^{\mu T}-1}{\mu} \right)T^{\frac{3p}{2}}e^{p\lambda_{1}T}\Upsilon^{2}  \leq~\frac{2\beta C(x)^p}{p}.
\end{align*}
Hence we have 
\begin{align*}
\left[T^{-\frac{3p}{2}}e^{(\mu - p\lambda_1)T} - T^{-\frac{3p}{2}}e^{-p\lambda_1T}\right] \leq \tilde C(x),
\end{align*}
for some function $\tilde C(x)$ independent of $T.$
Letting $T \to +\infty$ gives a contradiction, as $\mu > p\lambda_{1}$. Hence all solutions must blow-up in finite time.

%%%%%%%%%%%%%%%%%%%%%%%%%%%%%%%%%%%%%%%%%%%%%%%%%%%%%%%%%%%%%%%%%%%%%%%%%%%%%%%%%%%%%%%%%%%%%%%%%%%%%%%%%%%%%%%%%%%%%

\section{Proof of Theorem~\ref{thm2} : Global existence and Blow-up}\label{mainthm2}

We divide the proof in to two parts. The first part contains the question of global existence and the second part concentrated on the blow-up results.

\subsection{Proof of Theorem~\ref{thm2}}
\begin{proof}
We only need to prove the existence of a bounded global super-solution.

 Recall $\phi_1$ denotes the first eigenfunction of  $-\Delta_{\hn}$, which is bounded. We know that $w(x,t) = e^{-\lambda_1 t}\phi_1(x)$ is the unique solution to

\begin{equation} \label{eqn:heat}
\begin{cases}
&\partial_t w(x,t) = \Delta_{\hn} w(x,t) \hbox{ in }~ \mathbb{H}^{n}\times (0, \infty),\\
&w(x,0)= \phi_1(x) \in L^{\infty}(\hn).
\end{cases}
\end{equation}
Let $\varepsilon >0$ be small such that $g(s) =  s |\ln s|^{-\alpha},$ for $s \in (0,\varepsilon).$
Consider  two parameters $\delta$ and $\theta$ sufficiently small whose values will be chosen later. We look for a super-solution of the form $\bar u(x,t) = \theta e^{\delta t} w(x,t),$ where $w$ satisfies \eqref{eqn:heat}.

  We choose $\theta>0$ small enough so that $\sup_{x,t} \bar u(x,t) < \varepsilon.$ We note that it is sufficient to choose $\theta$ and $\delta$ satisfying 
\begin{align*}
0 \,<  \, \delta < \lambda_1, \ \ \ \ \ \theta ||\phi_1||_{L^{\infty}(\mathbb{H}^n)} < \varepsilon.
\end{align*}
Since the ground state is bounded and hence the above choices are feasible, and later we will make further smallness assumption on $\theta$ and $\delta.$ We assume the initial data in Theorem \ref{thm2} 
satisfies $u_0 \leq \theta \phi_1.$ We claim that the $\bar u$ is indeed a super-solution to \eqref{eqn:main} with $g$ given by \eqref{nonlinearity2} and the above choice of initial data. 
Plugging $\bar u$ in to the equation \eqref{eqn:main} we get 
\begin{equation*}
\partial_t \bar u(x,t) - \Delta_{\hn} \bar u(x,t) = \theta \delta  e^{\delta t}w(x,t) = \delta \bar u(x,t).
\end{equation*}
Therefore $\bar u$ is a super-solution to \eqref{eqn:main} with initial datum $u_0$ provided $\bar u$ satisfies 
$\delta \bar u(x,t) \geq t^q \bar u(x,t) (\ln (1/\bar u(x,t)))^{-\alpha}.$ A further simplification gives $\bar u$ must satisfy
\begin{equation}
\bar u(x,t) \leq e^{-\gamma t^{\frac{q}{\alpha}}}, \hbox{ where }~ \gamma = \delta^{-\frac{1}{\alpha}}.
\end{equation}
On the other hand, by our explicit expression of $\bar u$ we require that
\begin{align} \label{global}
\theta e^{\delta t + \gamma t^{\frac{q}{\alpha}} - \lambda_1 t} \phi_1(x) < \varepsilon.
\end{align}
The proof of global existence is now a consequence of the inequality \eqref{global}. 

\medskip

\noindent
{\bf Case 1:} $q < \alpha.$ Then since $\delta < \lambda_1,$ and $\phi_1$ is bounded, we can always choose $\theta,\delta$ satisfying \eqref{global} for all $t$. Hence $\bar u$ is a global super-solution. 

\medskip 

At the borderline case $q = \alpha$ the dimension of the underlying space, in particular, the value of $\lambda$ comes in to play.

\medskip

\noindent
{\bf Case 2:} $q = \alpha.$ Since $\gamma = \delta^{-\frac{1}{\alpha}}$, at the critical case $q = \alpha$ we need to choose $\delta$ satisfying $\delta + \delta^{-\frac{1}{\alpha}} \leq \lambda_1 = \frac{(n-1)^2}{4}.$ 
Note that $\delta + \delta^{-\frac{1}{\alpha}} \geq 1$ for sufficiently small  $\delta >0.$ If  $n \geq 4$, we can  simply take $\delta = 1,$ and choose $\theta$ small to obtain the existence of a global super solution.
This completes the proof of global existence $(a)$ and $(c)(ii).$ 
\end{proof}

\begin{remark}
It is important to note that  $\lambda_1 \leq 1$ in dimension $2$ and $3,$ hence  it is not possible to employ this strategy for such cases. In fact we will see in the next subsection that all the solutions blow-up in finite time
  when $q = \alpha$ in dimension $2.$ 
\end{remark}

\subsection{Proof of blow-up results}

For $u_{0}\in C(\mathbb{H}^{n})\cap L^{\infty}(\mathbb{H}^{n})$, let $T_{u_{0}}$ denote the maximal interval of existence of solution $u$ of \eqref{eqn:main}.  Let us proceed as in the previous section. We define for $T < T_{u_0}$ and $(x,t)\in \mathbb{H}^{n}\times [0,T)$
 \begin{align*}
\Phi(x, t) = \int_{\hn} \mathcal{K}(x,y,T- t) \, u(y, t)~{\rm d}v_{\mathbb{H}^{n}}(y).
\end{align*}
Moreover, it follows from \cite[Lemma~5.1]{Bandle}, the following estimate from below: $$\Phi(x,0)\geq C(x) T^{-3/2}e^{-\lambda_{1}T} := \tau(x,T).$$

Since $u$ is a bounded solution to \eqref{eqn:main}, we obtain 
\begin{align}\label{eqn:logid}
\Phi(x, t) = \Phi(x,0) + \int_0^t \int_{\hn} \mathcal{K}(x,y,T-s)f(u(y,s),s)~{\rm d}v_{\mathbb{H}^{n}}(y) \, {\rm d}s
\end{align}
where $f(s,t) = t^qg(s).$ 

\medskip 

Differentiating the identity \eqref{eqn:logid} and exploiting the convexity of $g$ we obtain  the inequality 
$$\partial_{t}(\Phi(x, t)) \geq t^qg(\Phi(x, t) ), \quad \mbox{for every} \   0< t \leq T.$$

Changing of variables and using the lower bound on $\Phi(x, 0)$ gives
\begin{align} \label{eqn:log1}
\int_{\tau(x, T)}^{\infty} \frac{{\rm d}\tau}{g(\tau)} 
\geq
\int_{\Phi(x, 0)}^{\Phi(x, t) } \frac{ {\rm d}\tau}{g(\tau)} \geq \frac{1}{q+1}t^{q+1},
\end{align}
Note that as $T \rightarrow \infty, \tau(x, T) \rightarrow 0+$ exponentially fast. Now we shall estimate the following integral $$\int_{\delta}^{\infty} \frac{d\tau}{g(\tau)}.$$

We know that $g(\tau) \geq \kappa \tau^2$ for $\tau \geq \frac{1}{2}$ and hence we have  

$$
\int_{\frac{1}{2}}^{\infty}\frac{{\rm d}\tau}{g(\tau)} \leq \frac{2}{\kappa} = O(1).
$$

Therefore for $\delta$ small
\begin{align} \label{eqn:log2}
\int_{\delta}^{\infty}\frac{ {\rm d}\tau}{g(\tau)} &\leq  \int_{\delta}^{\frac{1}{2}} \frac{ {\rm d}\tau}{\tau(\ln (\frac{1}{\tau}))^{-\alpha}} + O(1) \notag\\
&\leq \frac{1}{1+ \alpha} \left(\ln \left(\frac{1}{\delta}\right)\right)^{\alpha+1} + O(1).
\end{align}
Using the estimates \eqref{eqn:log1},  \eqref{eqn:log2} with $\delta = \tau(x, T)$ we get 
\begin{align*}
\frac{1}{q+1} T^{q+1} &\leq \frac{1}{\alpha + 1} \left(\ln \left(\frac{1}{C(x)T^{-\frac{3}{2}}e^{-\lambda_1T}}\right)\right)^{\alpha+1} + \, O(1)\\
&\leq \frac{1}{\alpha + 1} T^{\alpha+1} \left(\lambda_1 + \frac{1}{T}\ln \left(\frac{1}{C(x)T^{-\frac{3}{2}}}\right)\right)^{\alpha +1} + \, O(1).
\end{align*}
If $q > \alpha,$ the above identity can not hold true for every $ T > 0.$ On the other hand if $q = \alpha,$ then dividing by $ T^{q+1}$ and letting $T \rightarrow \infty$ we get $1 \leq \lambda_1^{\alpha +1}.$ In 
dimension $2$ it is not possible. Hence if $N = 2$ and $q=\alpha$ all solutions blow-up in finite time.
This completes the proof of $(b)$ and $(c)(i).$

\medskip
\begin{remark}
In dimension $3, \lambda_1 = 1.$ The case $p = \alpha$ in dimension $3$ remains open. 
\end{remark}

%%%%%%%%%%%%%%%%%%%%%%%%%%%%%%%%%%%%%%%%%%%%%%%%%%%%%%%%%%%%%%%%%%%%%%%%%%%%%%%%%%%%%%%%%%%%%%%%%%%%%%%%%%%%%%%%%%%%%%%%
\section{Proof of the theorems on Cartan Hadamard manifolds}\label{mainthm3and4}

\noindent
{\bf Proof of Global Existence:} Recall $\lambda_*(M)$ is defined by \eqref{lstarm}.

\medskip

\noindent
{\bf Proof of Theorem \ref{thm3}(a):}
We only need to exhibit a global, positive bounded super-solution to  problem \eqref{eqn:main} for small $\beta.$ Set 
\begin{align*}
H(t) = \int_0^t e^{(\mu - p\lambda_*(M)) \tau} {\rm d}\tau,   
\end{align*}
and $H_{\infty} = \lim_{t \rightarrow \infty} H(t),$ which is finite provided $\mu < p \lambda_*(M).$

By \cite[Theorem 3.2]{PF1}, for $ \mu < p \lambda_*(M)$, \eqref{eqn:3}
with $\hn$ replaced by $M,$ admits a global bounded classical solution $w(x,t)$ with initial data satisfying 
$u_0 \leq C \phi_M,$ where $C >0$ is a constant.
We choose  $\beta_{0}$ such that
\begin{align*}
0<\beta_{0}< \frac{1}{1+(\sup \limits_{ t>0}\|w(\cdot,t)\|_{L^{\infty}(M)})^p}.
\end{align*}
Then $w$ is a global bounded super-solution to \eqref{eqn:3}.

\medskip

\noindent
{\bf Proof of Theorem \ref{thm4}(a) and (c):}
Let $\varepsilon >0$ be such that $g(s) = s|\ln s|^{-\alpha}$ for $s \in [0,\varepsilon).$
Let $\phi_M$ be the bounded super solution to \eqref{evpM} with $\lambda_1(M)$ replaced by $\lambda_*(M)$. We set $w(x,t) = \theta e^{(\delta - \lambda_*(M))t} \phi_M(x),$ where $\delta < \lambda_*(M)$ and $\theta$ small such that  
\begin{align} \label{choiceof1}
\theta e^{\delta t + \gamma t^{\frac{q}{\alpha}} - \lambda_*(M) t} \phi_M(x) < \varepsilon, \ \ \gamma = \delta^{-\frac{1}{\alpha}}.
\end{align}
If $q < \alpha$ then it is always possible. On the other hand if $q = \alpha$ and $n \geq 1 + \frac{2\sqrt{2}}{\kappa},$ then $\lambda_*(M) \geq 2$ and we can choose $\delta = 1$ and $\theta$ small satisfying \eqref{choiceof1}.

\medskip

\noindent
{\bf Proof of Blow-up results.} It relies on the  following elementary lower bound on $e^{t\Delta_M}u_0,$ which is a consequence of ${\bf (H2)}.$

\begin{lemma} \label{lbm}\cite[Lemma~5.1]{PF2}
Let  $u_0 \in C(\hn) \cap L^{\infty}(\hn),$ $u_0 \not\equiv 0$ and $\varepsilon \in (0, \lambda_1(M)).$ Then   there exists a precompact set $\mathcal{O} \subset M,$ $t_0 >0$ and a constant $C=C(u_0, \mathcal{O}) >0$ depending on $u_0$ and
 $\mathcal{O}$ such that 
\begin{align*}
(e^{t\Delta_M}u_0)(x) \geq C e^{-(\lambda_1(M) + \varepsilon)t}, \ \ \ for \ x \in \mathcal{O} \ and \ t \geq t_0.
\end{align*}
\end{lemma}
\begin{proof}
Since $u_0 \not \equiv 0,$ there exists a precompact set $O$ such that vol$_M(O) < \infty$ and $u_0$ stays away from $0$ in $\mathcal{O}.$ By ${\bf (H2)}$, for a given $\varepsilon>0$ there exists $t_0$ such that 
\begin{align*}
\mathcal{K}_M(x,y,t) \geq \frac{1}{e^{(\lambda_1(M)+\varepsilon)t}}, \quad \mbox{for  all}  \ (x,y) \in \mathcal{O} \times \mathcal{O} \quad \mbox{and} \ t \geq t_0.
\end{align*}
Hence for $x \in \mathcal{O}$ and $t \geq t_0,$ 
\begin{align*}
e^{t\Delta_M}u_0(x) \geq \int_{\mathcal{O}}\mathcal{K}_M(x,y,t) \, u_0(y) \, {\rm d}v_M(y) \geq \frac{ \mbox{vol}_M(\mathcal{O})\inf_{\mathcal{O}} u_0}{e^{(\lambda_1(M)+\varepsilon)t}}.
\end{align*}
\end{proof}

\medskip

\noindent
{\bf Proof of Theorem \ref{thm3}(b)}
Fix $\varepsilon > 0$ such that $\mu > p(\lambda_{1}(M) + \varepsilon).$ As before setting 
\begin{align*}
\Phi(x,t):= \int \limits_{M}\mathcal{K}_M(x,y,T-t)u(y,t)~{\rm d}v_{M}(y)
\end{align*}
and following the same lines of the proof of Theorem \ref{thm1}(b), we get 
\begin{align} \label{upperbound1}
1\geq  \frac{p\beta^{k}}{(k-1)!}\left(\frac{e^{\mu t}-1}{\mu} \right)\Phi^{pk}(x,0),\quad \hbox{ for } 0\leq t\leq T \ and \ k \in \mathbb{N}.
\end{align}
Let $\mathcal{O}$ be the precompact set of Lemma \ref{lbm}. Then for every $x \in \mathcal{O},$ \eqref{upperbound1} and \eqref{lbm} provides
\begin{align} \label{zxc123}
&\frac{p\beta^{k}}{(k-1)!}\left(\frac{e^{\mu T}-1}{\mu} \right)e^{-pk(\lambda_{1}(M) + \varepsilon)T}C^{pk} \leq 1,  
\end{align}
Setting $\Upsilon = \beta C^pe^{-p(\lambda_{1}(M) + \varepsilon)T}$, multiplying by $\frac{1}{k(k+1)}$ and summing over all $k \in \mathbb{N}$ we obtain from \eqref{zxc123}
\begin{align*}
\left(\frac{e^{\mu T}-1}{\mu} \right)e^{p(\lambda_{1}(M) + \varepsilon)T}\Upsilon^{2}  \leq~\frac{2\beta C^p}{p}.
\end{align*}
Hence
\begin{align*}
\left[e^{(\mu - p(\lambda_{1}(M) + \varepsilon))T} - e^{-p\lambda_1(M)T}\right] \leq \tilde C,
\end{align*}
which is impossible to hold for every $T.$ 

\medskip

\noindent
{\bf Proof of Theorem \ref{thm4}(b):} As in \eqref{eqn:log1}, \eqref{eqn:log2} we have 
\begin{align*}
 \frac{t^{q+1}}{q+1} \leq \int_{\tau(t)}^{\infty} \frac{1}{g(s)} ds \leq \frac{1}{1+ \alpha} \left(\ln \left(\frac{1}{\tau(t)}\right)\right)^{\alpha+1} + O(1),
\end{align*}
where $\tau(t) = Ce^{-(\lambda_1(M) + \varepsilon)t}$ is the lower bound of $e^{t\Delta_M}u_0.$  Therefore 
\begin{align*}
 \frac{t^{q+1}}{q+1} \leq  \frac{t^{\alpha+1}}{\alpha+1}\left(\lambda_1(M) + \varepsilon - \frac{1}{t}\ln C \right)^{\alpha+1} + O(1). 
\end{align*}
Since $q > \alpha,$ letting $t \rightarrow \infty$ gives a contradiction.

%%%%%%%%%%%%%%%%%%%%%%%%%%%%%%%%%%%%%%%%

\section{Concluding Remarks}\label{concluding remarks}

We conclude this article with the following open question : 
\begin{align}\label{concluding_eqn:main}
\left\{\begin{array}{ll}
\partial_{t}u=\Delta_{\mathbb{H}^{3}} u+ f(u, t) &\hbox{ in }~ \mathbb{H}^{3}\times (0, T),\\
\\
\quad u =u_{0}\in C(\mathbb{H}^{3}) \cap L^{\infty}(\mathbb{H}^{3}) &\hbox{ in }~ \mathbb{H}^{3}\times \{0\},
\end{array}\right.
\end{align}
where $f$ is of the form \eqref{nonlinearity2} with $q = \alpha.$ Is there exists a global solution to \eqref{concluding_eqn:main} for small initial data or all non-negative, non-trivial solutions blow-up in finite time? The approach undertaken in this article is inadequate to capture the phenomenon in dimension $3.$
We believe that $q=\alpha$ might be a blow-up region, although we do not have any validation to our claim.

%%%%%%%%%%%%%%%%%%%%%%%%%%%%%%%%%%%%%%%%%

   \par\bigskip\noindent
\textbf{Acknowledgments.}
We are grateful to G.~Grillo for mentioning the references \cite{GMeP0, GMeP, GMP3}   to us 
and explaining some of  their results. 
D.~Ganguly is partially supported by the INSPIRE faculty fellowship (IFA17-MA98). D.~Karmakar acknowledges the support of the Department of Atomic Energy, Government of India, under project no. 12-R\&D-TFR-5.01-0520. S.~Mazumdar is partially supported by IIT Bombay SEED Grant RD/0519-IRCCSH0-025.

%%%%%%%%%%%%%%%%%%%%%%%%%%%%%%%%%%%%%%%%%%%%%%%%%%%%%%%%%%%%%%%%%%%%%%%%%%%%%%%%%%%%%%%%%%%%%%%%%%%%%%%%%%%%%%%%%%
%%%%%%%%%%%%%%%%%%%%%%%%%%%%%%%%%%%%%%%%%%%%%%%%%%%%%%%%%%%%%%%%%%%%%%%%%%%%%%%%%%%%%%%%%%%%%%%%%%%%%%%%%%%%%%%%%%

\end{document}